\documentclass[11pt,final,reqno]{amsart}
\usepackage{fullpage}
\usepackage{amsmath,amsfonts,amssymb,amsthm}
\usepackage{graphicx}
\usepackage{epstopdf}
\usepackage{color}
\usepackage{epsf}
\usepackage{verbatim}
\usepackage{enumitem}
\usepackage{hyperref}
\usepackage[labelformat=empty]{subfig}

\newtheorem{theorem}{Theorem}[section]
\newtheorem{lemma}[theorem]{Lemma}
\newtheorem{proposition}[theorem]{Proposition}
\newtheorem{remark}[theorem]{Remark}

\newtheorem{definition}[theorem]{Definition}

\newtheorem{claim}[theorem]{Claim}
\newtheorem{fact}[theorem]{Fact}

\newcommand{\bbE}{{\mathbb E}}
\newcommand{\bbH}{{\mathbb H}}

\newcommand{\cA}{{\mathcal A}}
\newcommand{\cC}{{\mathcal C}}
\newcommand{\cM}{{\mathcal M}}
\newcommand{\cQ}{{\mathcal Q}}
\newcommand{\cR}{{\mathcal R}}
\newcommand{\cS}{{\mathcal S}}
\newcommand{\cT}{{\mathcal T}}

\newcommand{\del}{\delta}
\newcommand{\Del}{\Delta}
\newcommand{\ep}{\varepsilon}
\newcommand{\e}{\varepsilon}
\newcommand{\alp}{\alpha}
\newcommand{\eps}{\varepsilon}
\newcommand{\g}{\gamma}
\newcommand{\sig}{\sigma}

\title{Tiling $3$-uniform hypergraphs with $K_4^3-2e$}
\author{Andrzej Czygrinow}\address{School of Mathematical and Statistical Sciences, Arizona State University, Tempe, AZ 85287, USA.}
\email{andrzej.czygrinow@asu.edu}
\author{Louis DeBiasio}
\address{Department of Mathematics, Miami University, Oxford, OH 45056, USA.}
\email{debiasld@muohio.edu}
\author[B.~Nagle]{Brendan Nagle}
\thanks{The third author was partially supported by NSF grant DMS~1001781.}
\address{Department of Mathematics and Statistics,
University of South Florida, 4202 E.~Fowler Ave, PHY~144, Tampa, FL 33620--5700, USA.}
\email{bnagle@usf.edu}

\begin{document}

\begin{abstract}
Let $K_4^3- 2e$ denote the hypergraph consisting of two triples on four points.  
For an integer $n$, let $t(n, K_4^3-2e)$ denote the smallest integer $d$ so that 
every 3-uniform hypergraph $G$ of order $n$ with minimum pair-degree $\delta_2 (G) \geq d$ contains $\lfloor n/4\rfloor $
vertex-disjoint copies of $K_4^3-2e$.  
K\"uhn and Osthus~\cite{KO} proved that 
$t(n, K_4^3-2e)  = \tfrac{n}{4}(1 + o(1))$ holds for large integers $n$.  
Here, we prove the exact counterpart, that for all sufficiently large integers $n$ divisible by 4, 
$$
t(n, K_4^3-2e) = 
\left\{ 
\begin{array}{cc}
\tfrac{n}{4} & \text{when $\tfrac{n}{4}$ is odd,} \\
\tfrac{n}{4} + 1 & \text{when $\tfrac{n}{4}$ is even.}
\end{array}
\right.  
$$
A main ingredient in our proof is the recent 
`absorption technique' 
of R\"odl, Ruci\'nski and Szemer\'edi.  
\end{abstract}

\maketitle

\section{Introduction}
For a fixed $k$-graph $H_0$ of order $m$, 
we say that 
a given 
$k$-graph 
$G$ of order $n$ 
is {\it $H_0$-tileable}
if $G$ contains, as subhypergraphs, $\lfloor n/m \rfloor$ vertex-disjoint copies of $H_0$.  
Now, suppose $G$ has vertex set $V$, 
and for an integer $1\leq \ell \leq k$, let $U \in \tbinom{V}{\ell}$ be given.  As is customary, let 
$$
N(U) = N_G(U) = \left\{W \in \binom{V}{k- \ell}: U \cup W \in E(G)\right\}\, \quad \text{and} \quad 
\del_{\ell}(G) = \min \left\{ |N(U)|: U \in \binom{V}{\ell}\right\}
$$
denote, respectively, the neighborhood of $U$ in $G$, and the {\it $\ell$-degree} of $G$.  
Define $t_{\ell}^k(n, H_0)$ to be the smallest integer $d$ so that every $k$-graph $G$ of order $n$
for which $\del_{\ell}(G) \geq d$
holds is $H_0$-tileable.

In the case of graphs ($k=2$), $t_1^2(n, H_0)$ is known, up to an additive constant, 
for every fixed graph $H_0$ (see \cite{KO2}).  Furthermore, there
are some graphs $H_0$ for which $t_1^2(n, H_0)$ is known exactly.  
The most celebrated such result is the Hajnal-Szemer\'edi theorem~\cite{HSz}, which says that for the $r$-clique $H_0 = K_r$ and for $n$ divisible by $r$, 
$$
t_1^2(n, K_r) = \left(1 - \frac{1}{r}\right)n.  
$$
A recent result of Wang~\cite{W} shows that for all integers $n$ divisible by 4, 
$t_1^2(n, C_{4}) = 
\tfrac{n}{2}$.
This result is a special case of the well-known El-Zahar conjecture, and had been independently conjectured by Erd\H{o}s and Faudree.  


In the case of hypergraphs ($k\geq 3$), 
much less is known about tiling problems.  
For only 
the $k$-edge $H_0 = K_k^k$ (the tiling of which is a perfect matching) 
is 
$t_{k-1}^k(n, H_0)$ known for all $k \geq 3$.  
This significant result is due to 
R\"odl, Ruci\'nski and Szemer\'edi~\cite{RRS}, and asserts that for all sufficiently large integers $n$ divisible by $k$, 
$$
t_{k-1}^k(n, K_k^k)  = 
\frac{n}{2} - k + \eps_{k,n}, \quad \text{where} \quad 
\eps_{k,n} \in 
\left\{ \frac{3}{2}, \, 2,\,  \frac{5}{2}, \,  3\right\}
$$
is determined by explicit divisibility conditions on $n$ and $k$.

We are interested in tilings when $k=3$ and $\ell=2$, where some 
interesting 
results have recently developed.  
(In what follows, we abbreviate $t_2^3(n,H_0)$ to $t(n,H_0)$.)  
As usual, 
let $K_4^3$ denote the complete 3-graph on 4 vertices.  Let $K_4^3 - e$ denote its subhypergraph consisting of 3 edges, and let $K_4^3 - 2e$ denote its subhypergraph consisting of 2 edges.  
K\"uhn and Osthus~\cite{KO} proved that 
$t(n, K_4^3 - 2e) = (1 +o(1)) n/4$.  Recently, 
Lo and Markstr\"om~\cite{LM1, LM2}
have shown that 
$t(n, K_4^3 -e) = (1 + o(1)) n/2$
and that 
$t(n, K_4^3) = (1 + o(1)) 3n/4$.
Keevash and Mycroft~\cite{KM} showed the exact counterpart that, 
for sufficiently large integers $n$ divisible by 4, 
$t(n, K_4^3) = (3n/4) - \eps_n$, 
where $\eps_n = 2$ if $8|n$ and $\eps_n = 1$ otherwise.  
We shall prove the following exact result for $K_4^3 - 2e$.  

\begin{theorem}\label{main}
For all sufficiently large integers $n$ divisible by 4,  
$$
t(n, K_4^3-2e) = 
\left\{
\begin{array}{cc}
\frac{n}{4} & \text{when $\tfrac{n}{4}$ is odd,}  \\
\frac{n}{4} + 1 & \text{when $\tfrac{n}{4}$ is even.}  
\end{array}
\right.  
$$
\end{theorem}

\noindent The proof of Theorem~\ref{main}
spans 
Sections~\ref{section:extreme}
and~\ref{section:non-extreme}.  
We mention that an essential ingredient in our proof is the `absorption technique' (see 
Section~\ref{section:non-extreme}) 
of R\"odl, Ruci\'nski and Szemer\'edi.

In the remainder of this 
paper, we shall make the abbreviation $D = K_4^3 - 2e$.  
(In the papers~\cite{KO, RRS}, $D = K_4^3 - 2e$ was abbreviated by $\cC$ and $C_4^{3,1}$, respectively, since for those authors, $K_4^3-2e$ was viewed as a type of cycle.)  
In the remainder of this 
introduction, we discuss the main concept used in the proof of Theorem~\ref{main}, that of an `$\eps$-extremal' 3-graph (for $D= K_4^3-2e$).

\subsection{Theorem~\ref{main} and~$\eps$-extremal 3-graphs}

To motivate the concept of an $\eps$-extremal 3-graph (stated in the upcoming Definition~\ref{extremaldef}), 
we first observe the following constructions for the lower bounds of Theorem~\ref{main}.

Let $A$ be a set of $\tfrac{n}{4} - 1$ vertices, and let $B$ be a set of $\tfrac{3n}{4} + 1$ additional vertices.  
Define $G_0 = \tbinom{A \cup B}{3} \setminus \tbinom{B}{3}$, and note that $\del_2(G_0) = \tfrac{n}{4} - 1$.  
When $\tfrac{n}{4}$ is even, add any Steiner triple system\footnote{A {\it Steiner triple system} (STS) is a 
3-graph $H$ where $\del_2(H) = \Del_2(H) = 1$. 
It is well-known that an STS of order $m$ exists if, and only if, $m \equiv 1, 3$ (mod 6).}  
on vertex set $B$ to $G_0$, and call this hypergraph $G_1$, where we note that $\del_2(G_1) = \tfrac{n}{4}$.  
Since $G_i[B]$, $i = 0, 1$, is $D$-free,  
every copy of $D$ in $G_i$
contains at least one vertex of $A$, and so 
$G_i$ is not 
$D$-tileable.  

\begin{definition}[$\e$-extremal]  
\label{extremaldef}
\rm 
Let $\e> 0$ be given, and 
suppose $G$ is a 3-graph 
of order $n$.  
We say $G$ is {\it $\e$-extremal}
if there exists $S \subset V(G)$ of size $|S| \geq (1 - \ep) \tfrac{3n}{4}$ for which $G[S]$ is $D$-free.  
\end{definition}

While the lower bound
constructions for Theorem~\ref{main} 
motivate the concept of Definition~\ref{extremaldef}, the following fact indicates why we choose the terminology `extremal'.

\begin{fact}
\label{noSTS}
Let $G$ be a 3-graph 
on $n$ vertices, where $n$ is divisible by 4, satisfying 
\begin{equation}
\label{eqn:codeg}
\del_2(G) \geq 
\left\{
\begin{array}{cc}
\frac{n}{4} & \text{when $\tfrac{n}{4}$ is odd,} \\
\frac{n}{4} + 1 & \text{when $\tfrac{n}{4}$ is even.} 
\end{array}
\right.  
\end{equation}
Then any $S \subset V(G)$ for which $G[S]$ is $D$-free satisfies $|S| \leq \tfrac{3}{4}n$.  
\end{fact}  

\begin{proof} Since  
$G[S]$ is $D$-free, when $\tfrac{n}{4}$ is even, 
we have $\tfrac{n}{4} + 1 \leq \del_2(G) \leq n - (|S| - 1)$, and the result follows.  
When 
$\tfrac{n}{4}$ is odd, 
suppose 
some $S_0 \subset V(G)$ exists of size $\tfrac{3n}{4} + 1$ for which $G[S_0]$ is $D$-free.  
Since $G[S_0]$ is not an STS (since $\tfrac{3n}{4} + 1 \not\equiv 1, 3$ (mod 6)), some pair $s, s' \in S_0$ satisfies 
$N(s, s') \cap S_0 = \emptyset$, in which case $\frac{n}{4} \leq |N(s, s')| \leq n - |S_0|$, and the result follows.  
\end{proof}

Now, the upper bounds in Theorem~\ref{main} follow immediately from the following two statements.

\begin{theorem}[Theorem~\ref{main} -- extremal case]  
\label{extreme}
There exists $\e_0 > 0$ so that, 
for all sufficiently large integers $n$ divisible by 4, the following holds.  
Whenever $G$ is a 3-graph 
of order $n$ 
satisfying~(\ref{eqn:codeg}) and which 
is $\e_0$-extremal, then $G$ is $D$-tileable.  
\end{theorem}  

\noindent We prove Theorem~\ref{extreme} in Section~\ref{section:extreme}.  

\begin{theorem}[Theorem~\ref{main} -- non-extremal case]  
\label{non-extreme}
For every $\ep>0$ 
and for all sufficiently large integers $n$ divisible by 4, the following holds.  
Whenever $G$ is a 3-graph of order $n$ satisfying~(\ref{eqn:codeg})  
(see Remark~\ref{remark}),   
which is not $\e$-extremal, then $G$ is $D$-tileable.  
\end{theorem}  

\noindent We prove Theorem~\ref{non-extreme} in Section~\ref{section:non-extreme}.  

\begin{remark}
\label{remark}  
\rm 
We mention that Theorem~\ref{non-extreme} can be proved, for the same money, under a slightly weaker hypothesis than~(\ref{eqn:codeg}).  
In particular, 
Theorem~\ref{non-extreme} remains valid if one only assumes 
that 
$\del_2(G) \geq (n/4) (1 - \gamma)$, for a constant $\gamma > 0$ sufficiently smaller than $\eps$.  
\end{remark}

\section{Proof of Theorem~\ref{extreme}}
\label{section:extreme}

We shall use the following theorem of Pikhurko \cite{P}, stated here in a less general form.

\begin{theorem}[\cite{P}, Theorem 3]\label{4graph}
Let $H$ be a $4$-partite $4$-graph with 4-partition $V(H) = V_1\cup V_2\cup V_3\cup V_4$, where $|V_1| = \dots = |V_4| = m$.  
Let $\delta(V_1)=\min\{|N(v_1)|: v_1\in V_1\}$ and 
$$
\delta(V_2, V_3, V_4)=\min\{|N(v_2, v_3, v_4)|: v_2\in V_2, \, v_3\in V_3, \, v_4\in V_4\}.
$$  
For $\g > 0$ and a sufficiently large integer $m$, if 
$$
m\delta(V_1)+m^3 \delta(V_2, V_3, V_4)\geq (1+\gamma)m^4,
$$
then $H$ contains a perfect matching.
\end{theorem}

To prove Theorem~\ref{extreme}, it suffices to take $\e_0 = 10^{-18}$, and we shall take $n$ sufficiently large, whenever needed.  
We write $n = 4k$
and $\alp^3 = \e_0$.  
Let $G$ be a 3-graph of order $n$ satisfying~(\ref{eqn:codeg}) 
which is $\e_0$-extremal.  We prove that $G$ is $D$-tileable, and will construct a $D$-tiling in stages.  In particular, we will build
vertex-disjoint  partial $D$-tilings $\cQ$, $\cR$, $\cS$ and $\cT$ whose union is a $D$-tiling of $G$.  
To build these partial tilings, we need a few initial considerations.

To begin, let $Z \subset V(G)$ be a maximal set for which $G[Z]$ is $D$-free.  
Define 
\begin{equation}
\label{eqn:X} 
X = \left\{x \in V(G) \setminus Z: \left|N(x) \cap \binom{Z}{2}\right| \geq (1 - \alp) \binom{|Z|}{2} \right\}, 
\end{equation}  
and $Y = V(G) \setminus (X \cup Z)$.  
We estimate each of the quantities in $|X| + |Y| + |Z| = 4k = n$:  
\begin{equation}
\label{Yupper}  
k(1 - 4\alp^2) 
\leq 
|X|
\leq 
k (1 + 3\e_0), \quad 
0 \leq |Y| \leq 4\alpha^2 k, \quad 
3k(1 - \e_0) \leq 
|Z|
\leq 
3k,   
\end{equation}
i.e., $|Y|$ is small, $|X|$ is around $n/4$ and $|Z|$ is around $3n/4$.   
Indeed, the third estimate in~(\ref{Yupper}) follows from our hypothesis and Fact~\ref{noSTS}.  
To see the second estimate, 
for 
$W \subset X \cup Y$, 
write 
$G[Z, Z, W]$ for the 
collection of triples of $G$ consisting of two vertices from $Z$ and one vertex from $W$.  Then, 
$$
(k - 1) \binom{|Z|}{2} \leq \big|G[Z,Z,X\cup Y]\big| 
\leq 
(1 - \alp) \binom{|Z|}{2} |Y| + \binom{|Z|}{2} |X|, 
$$
so that 
$k-1+\alpha|Y|\leq |X|+|Y|$.  
The estimate on $|Z|$ implies that 
$|X| + |Y| \leq k + 3\e_0 k$, and so we have
the second estimate of~(\ref{Yupper}).    
Finally, our bounds on $|Y|$ and $|Z|$ render the first estimate in~(\ref{Yupper}).

Let us also check that~\eqref{Yupper} implies that 
\begin{equation}\label{zzx}
\forall z_1, z_2\in Z, \ |N(z_1, z_2)\cap X|\geq (1-\alpha)|X|.
\end{equation}
Indeed, 
since $|N(z_1, z_2) \cap Z| \leq 1$, we have 
$$
|N(z_1, z_2) \cap X| \geq k - 1 - |Y| 
\stackrel{\text{(\ref{Yupper})}}{\geq}  
(1 - 5 \alp^2)  k 
\stackrel{\text{(\ref{Yupper})}}{\geq}  
\frac{1 - 5 \alp^2}{1 + 3\e_0}  |X|  
\geq 
(1 - \alp) |X|.  
$$
We now introduce the first of our partial $D$-tilings, namely, $\cQ$.  \\

\noindent {\bf The partial $D$-tiling $\cQ$.}  
Let $\cQ$ be a largest 
$D$-tiling in $G$ for which each element $D_0 \in \cQ$ has three vertices in $Z$ and one vertex in $Y$.  Write $q = |\cQ|$, 
write $Y_{\cQ} \subset Y$ for the set of vertices of $Y$ covered by $\cQ$, 
and write $Z_{\cQ} \subset Z$ for the set of vertices of $Z$ covered by $\cQ$.  Clearly, $|Y_{\cQ}| = q$ and $|Z_{\cQ}| = 3q$.  
Write $\ell = k - |X|$, where we note 
from~(\ref{Yupper}) that   
\begin{equation}
\label{eqn:sizeofell}
-3 \e_0 k
\leq 
\ell = k - |X| \leq 4 \alpha^2 k.    
\end{equation}
For future reference, we make the following two claims.    

\begin{claim}
$q\geq \ell = k - |X|$.  
\end{claim}

\begin{proof}
If $\ell\leq 0$, there is nothing to show.  
If $\ell=1$, we have $|Y\cup Z|=3k+1$, and thus Fact~\ref{noSTS} 
implies that $G[Y\cup Z]$ contains a copy of $D$, which requires $|Y| \geq 1$.  
Now, if $q = 0$, then we could move a vertex from $Y$ to $Z$, which contradicts the maximality of $Z$.  
Finally, suppose $\ell\geq 2$, 
and observe that 
the quantity 
$|G[Z,Z,Y]| 
= 
|G[Z,Z,Y_{\cQ}]| + |G[Z,Z,Y \setminus Y_{\cQ}]|$ satisfies 
that 
\begin{multline*}  
(\ell - 1) \binom{|Z|}{2} 
\leq |G[Z,Z,Y]|  
\leq 
|Y_{\cQ}| (1 - \alp) \binom{|Z|}{2} + \left(\frac{|Z| - |Z_{\cQ}|}{2} + |Z_{\cQ}| |Z|\right)\big|Y\setminus Y_{\cQ} \big|  \\
= 
q (1 - \alp) \binom{|Z|}{2} + \left(\frac{|Z| - 3q}{2} + 3q |Z|\right)\big(|Y| - q\big) 
\stackrel{\text{(\ref{Yupper})}}{\leq}
q (1 - \alp) \binom{|Z|}{2} + 16 \alp^2 q |Z| k.  
\end{multline*}  
Now, if $q \leq \ell - 1$, then 
$$
1 \leq 1 - \alp + 32\frac{\alp^2 k}{|Z| - 1} 
\stackrel{\text{(\ref{Yupper})}}{\leq}
1 - \alp + 16 \alp^2, 
$$
a contradiction.  
\end{proof}

Note that, on account of the claim above, 
\begin{equation}
\label{eqn:sizeofq-ell}  
0 \leq q - \ell 
\stackrel{\text{(\ref{eqn:sizeofell})}}{\leq} 
|Y| + 3\e_0 k
\leq 
|Y| + 4\alpha^2 k
\stackrel{\text{(\ref{Yupper})}}{\leq} 
8 \alpha^2 k.  
\end{equation}

\begin{claim}\label{yzx}
For all $y\in Y\setminus Y_{\cQ}$ and $z\in Z\setminus Z_{\cQ}$, $|N(y,z)\cap X|\geq (1-\alpha)|X|$.  
\end{claim}

\begin{proof}
Fix $y\in Y\setminus Y_{\cQ}$ and $z\in Z\setminus Z_{\cQ}$.  By the maximality of $\cQ$, we have $|N(y,z)\cap Z|\leq |Z_{\cQ}| +1 = 3q + 1$.  
As such, since $|Y| \geq q$, 
we have 
$$
|N(y, z) \cap X| \geq k - (3q + 1) - (|Y| - 1) 
\geq k - 4|Y| 
\stackrel{\text{(\ref{Yupper})}}{\geq} 
(1 - 16 \alp^2 ) k 
\stackrel{\text{(\ref{Yupper})}}{\geq} 
\frac{1 - 16 \alp^2}{1 + 3\e_0} |X| 
\geq (1 - \alp)|X|.  
$$
\end{proof}

\noindent {\bf The partial $D$-tiling $\cR$.}  
We now use~\eqref{zzx} and Claim~\ref{yzx} to build 
a collection $\cR$ of 
$|Y\setminus Y_{\cQ}|$ 
vertex-disjoint copies of $D$, each with $1$ vertex in $Y\setminus Y_{\cQ}$, $1$ vertex 
in $X$, and two vertices in $Z\setminus Z_{\cQ}$.  
For sake of argument, 
assume $|Y \setminus Y_{\cQ}| \geq 1$.
Inductively, assume we have obtained $0 \leq i < |Y\setminus Y_{\cQ}|$ 
vertex-disjoint copies of $D$, each with $1$ vertex in $Y\setminus Y_{\cQ}$, $1$ vertex 
in $X$, and two vertices in $Z\setminus Z_{\cQ}$.  
Arbitrarily 
select an uncovered $y' \in Y\setminus Y_{\cQ}$
and uncovered $z_1', z_2' \in Z \setminus Z_{\cQ}$, noting that the latter is possible 
since at most $|Z_{\cQ}| + 2i \leq 5|Y| \leq |Z| - 2$ 
(cf.~(\ref{Yupper})) 
vertices
in $Z$ are unavailable for selection.  
Since $|N(y', z_1') \cap N(z_1', z_2') \cap X| \geq (1 - 2\alp)|X|$, 
we have at least $(1 - 2\alp) |X| - i \geq (1 - 2\alp ) |X| - |Y| > 0$ (cf.~(\ref{Yupper})) choices for an uncovered
vertex $x' \in X$, to complete the $(i+1)^{\rm st}$ copy of $D$.

Note that all vertices of $Y$ are covered by  $\cQ$ or $\cR$.  Let $Z_{\cQ,\cR} \supset Z_{\cQ}$ denote the set of vertices of $Z$ covered by $\cQ$ or $\cR$, 
and let $X_{\cR}$ denote the set of vertices of $X$ covered by $\cR$ (no vertices of $X$ were covered by $\cQ$).  
Observe that 
$$
|X \setminus X_{\cR}| = |X| - (|Y| - |Y_{\cQ}|) 
= k - |Y| + (q - \ell), \text{ and}   
$$
\begin{equation}
\label{eqn:sizeofZ-ZQR}  
|Z \setminus Z_{\cQ, \cR}|=  |Z| - |Z_{\cQ}|-2(|Y|-|Y_{\cQ}|) = 3(k-|Y|)-(q-\ell), 
\end{equation}
where we used that $|Z| = 4k - |X| - |Y| = 3k + \ell - |Y|$.   \\

\noindent {\bf The partial $D$-tiling $\cS$.}  
We now obtain a collection $\cS$ of $q - \ell$ vertex-disjoint copies of $D$, each with 2 vertices in $X \setminus X_{\cR}$ and 2 vertices in $Z \setminus Z_{\cQ, \cR}$.  
Indeed, arbitrarily pick vertices $z_1, z_1', \dots, z_{q-\ell}, z_{q - \ell}' \in Z \setminus Z_{\cQ, \cR}$, which is possible since 
$$
|Z \setminus Z_{\cQ, \cR}|
- 2 (q - \ell) 
\stackrel{(\ref{eqn:sizeofZ-ZQR})}{=}  
3 (k - |Y| - (q - \ell))  
\stackrel{\text{
(\ref{Yupper}), 
(\ref{eqn:sizeofq-ell})}}{\geq} 
3 k (1 - 12 \alpha^2)  \geq 2.  
$$
Inductively, assume we have covered 
$0 \leq i < q - \ell$ 
pairs $z_1, z_1', \dots, z_i, z_i'$
by vertex-disjoint copies $D_1, \dots, D_i$ of $D$, where each $D_j$, $0 \leq j \leq i$, has vertices 
$\{z_j, z_j', x_j, x_j'\}$, where $x_j, x_j' \in X \setminus X_{\cR}$.  
We infer from~\eqref{zzx} that 
$$
\big|N(z_1, z_1') \cap \big(X \setminus (X_{\cR} \cup \{x_1, x_1',\dots, x_i, x_i'\})\big)\big| \geq (1 - \alp) |X| - |X_{\cR}| - 2i 
\geq (1 - \alp) |X| - |Y| - 2(q - \ell) 
\geq 2, 
$$
where the last inequality holds on account of~(\ref{Yupper})
and~(\ref{eqn:sizeofq-ell}).  
We thus obtain 
the $(i+1)^{\rm st}$ copy of $D$.

Let $Z_{\cQ,\cR, \cS} \supset Z_{\cQ, \cR}$ denote the set of vertices of $Z$ covered by $\cQ$, $\cR$ or $\cS$, and let $X_{\cR, \cS} \supset X_{\cR}$ denote the set of vertices of $X$ covered by $\cR$ or $\cS$.    
Set $m:= |X \setminus X_{\cR,\cS}|$ and note that 
\begin{equation}
\label{ZQRS}  
m = 
|X \setminus X_{\cR, \cS}| 
\stackrel{\text{(\ref{eqn:sizeofZ-ZQR})}}{=}  
k - |Y| - (q - \ell) \quad \text{and} \quad |Z \setminus Z_{\cQ, \cR, \cS}| 
\stackrel{\text{(\ref{eqn:sizeofZ-ZQR})}}{=}  
3\big(k- |Y| - (q - \ell)\big) = 3m.  
\end{equation}
We conclude the proof of Theorem~\ref{extreme} by building the remaining partial $D$-tiling $\cT$.   \\

\noindent {\bf The partial $D$-tiling $\cT$.}  
Arbitrarily partition $Z \setminus Z_{\cQ, \cR, \cS} = Z_1 \cup Z_2 \cup Z_3$ into three sets of size $m$, and for simplicity of notation, write $X_0 = X \setminus X_{\cR, \cS}$.      
Define the following auxiliary 4-partite 4-graph $H$ with 4-partition $V(H) = X_0 \cup Z_1 \cup Z_2 \cup Z_3$, obtained by including each $\{x, z_1, z_2, z_3\} \in H$, $x \in X_0$, $z_i \in Z_i$ for $i = 1, 2, 3$,  
if 
$\{x, z_1, z_2, z_3\}$ spans a copy of $D$ in $G$.  
We claim that $H$ satisfies the hypothesis of 
Theorem~\ref{4graph} with $\g = 1/2$, and hence contains a perfect matching, which will then define $\cT$ and finish our proof
of Theorem~\ref{extreme}.

To bound $\del_H(Z_1, Z_2, Z_3)$, 
fix $z_1 \in Z_1, z_2 \in Z_2, z_3 \in Z_3$.  
We infer from~\eqref{zzx} that 
\begin{multline*}  
|N_H(z_1, z_2, z_3)| \geq \big| N_G (z_1, z_2) \cap N_G (z_1, z_3) \cap X_0 \big| \geq (1 - 2\alp)|X| - |X_{\cR, \cS}|  \\
\geq 
(1 - 2\alp) |X| - |Y| - 2(q - \ell) 
\stackrel{\text{(\ref{Yupper}), (\ref{eqn:sizeofq-ell})}}{\geq} 
(1 - 2\alp)|X| - 20 \alp^2 k   
\stackrel{\text{(\ref{Yupper})}}{\geq}
\big((1 -2\alp)((1 - 4\alp^2) - 20 \alp^2\big)) k  \\
\stackrel{\text{(\ref{Yupper})}}{\geq}
\frac{1 - 26 \alp}{1 + 3\e_0} |X|  
\geq (1 - 27 \alp) |X| \geq (1 - 27 \alp) |X_0|  
= (1 - 27 \alp) m.  
\end{multline*}  
Thus, $\del_H(Z_1, Z_2, Z_3) \geq (1 - 27 \alp)m$.

To bound $\del_H(X_0)$, 
fix $x \in X_0$, and 
for clarity of notation in what follows, write $N_G(x) = G_x$.  
By the definition of $X$, we have that $|G_x[Z]| \geq (1 - \alp) \tbinom{|Z|}{2}$, and so all but at most $\sqrt{\alp} |Z|$ vertices 
$z \in Z$ satisfy that $\deg_{G_x[Z]}(z) \geq (1 - \sqrt{\alp}) |Z|$.  
For each such $z \in Z$ and $i = 1, 2, 3$, 
$$
|N_{G_x}(z) \cap Z_i| \geq (1 - \sqrt{\alp})|Z| - |Z_{\cQ, \cR, \cS}| - 2m  
\stackrel{\text{(\ref{ZQRS})}}{=} 
m - \sqrt{\alp} |Z|   
= 
\left(1 - \sqrt{\alp}\frac{|Z|}{m}\right)m.  
$$
Since, by~(\ref{Yupper}) and~(\ref{ZQRS}), we have 
\begin{equation}
\label{m}  
3m
= 
|Z| - |Z_{\cQ, \cR, \cS}|
= 
|Z| - \big(3q + 2(|Y| - q) + 2(q - \ell)\big) 
\geq 
|Z| - 5|Y| + 2 \ell 
\stackrel{\text{(\ref{Yupper}), (\ref{eqn:sizeofell})}}{\geq}  
|Z| - 26 \alp^2 k 
\stackrel{\text{(\ref{Yupper})}}{\geq}  
\frac{|Z|}{2}, 
\end{equation} 
we conclude that 
$$
|N_{G_x}(z) \cap Z_i| \geq (1 - 6\sqrt{\alp}) m.  
$$
As such, 
$$
|N_H(x)|
\geq \sum_{z_1 \in Z_1} 
|N_{G_x}(z_1) \cap Z_2| 
|N_{G_x}(z_1) \cap Z_3| 
\geq 
\left(m - \sqrt{\alp}|Z|\right)
\left(\left(1 - 6\sqrt{\alp}\right) m\right)^2
\stackrel{\text{(\ref{m})}}{\geq}  
\left(1 - 6\sqrt{\alp}\right)^3 m^3,  
$$
and so $\del_H(X_0) \geq (1 - 234\sqrt{\alp}) m^3$.

The obtained bounds on $\del_H(Z_1, Z_2, Z_3)$ and $\del_H(X_0)$ then implies 
$$
m \del_H(X_0) + m^3 \del_H(Z_1, Z_3, Z_3) \geq 
\left(2 - 234\sqrt{\alp} - 
27 \alp
\right) m^4 
\geq 
\left(2 - 261\sqrt{\alp} 
\right) m^4 
> 
\frac{3}{2} m^4 
$$
so that, as claimed, 
$H$ satisfies the hypothesis of 
Theorem~\ref{4graph} with $\g = 1/2$.

\section{Proof of Theorem~\ref{non-extreme}}
\label{section:non-extreme}

Our proof of Theorem~\ref{non-extreme} is based on the following two lemmas, the second of which mirrors an `absorption' lemma of R\"odl, Ruci\'nski and Szemer\'edi~\cite{RRS}.  

\begin{lemma}\label{almost-perfect}
For all $\g > 0$ and sufficiently large integers $m$ divisible by $4$, the following holds.  Let $H$ be a 3-graph of order $m$.  If $\del_2(H) \geq \left(\tfrac{1}{4} - \g\right) m$ and $H$ is not $(8\g)$-extremal, then $H$ admits a $D$-tiling covering all but $50/\g$ vertices.  
\end{lemma}

\begin{lemma}
\label{absorbing-lemma}
For all $\alpha >0$
and sufficiently large integers $n$ divisible by 4, the following holds.  Let $G$ be a 3-graph of order $n$.  If $\del_2(G) \geq n/4$, then there exists $A \subset V(G)$ of size $|A| \leq \alp n$ so that,  
for every $W \subset V \setminus A$ of size $|W| \leq  50/\alp$ 
for which $|A \cup W|$ is divisible by 4, 
the hypergraph $G[A \cup W]$ is $D$-tileable.  
\end{lemma}

\noindent We defer the proofs of 
Lemmas \ref{almost-perfect} and \ref{absorbing-lemma} to Sections~\ref{sec:almost-perfect} and~\ref{sec:absorbing-lemma} respectively 
in favor of first proving Theorem~\ref{non-extreme}.  

\begin{proof}[Proof of Theorem \ref{non-extreme}]

Let $\eps > 0$ be given, together with a sufficiently large integer $n$ which is divisible by 4.  
Let $G$ be a $3$-graph of order $n$ satisfying~(\ref{eqn:codeg}) which is not $\e$-extremal.  
For $\alp =  \eps / 9$, let $A \subset V(G)$ be the set given by Lemma~\ref{absorbing-lemma}.  
Set $H = G[V \setminus A]$, 
and write $m = n - |A|$.

We claim that $H$ satisfies the hypothesis of Lemma~\ref{almost-perfect} with $\g = \alp$.  Indeed, 
note that 
$$
\del_2(H) \geq \frac{n}{4} - |A| \geq \frac{n}{4} - \alp n = \left(\frac{1}{4} - \alp\right)n \geq
\left(\frac{1}{4} - \alp\right)m.  
$$
Observe, moreover, that $H$ is not $(8\alp)$-extremal.  Indeed, if $S \subset V(H)$ satisfies that $H[S]$ is $D$-free, then $G[S]$ is also $D$-free, and   
if 
$$
|S| \geq 
(1 - 8\alp) \frac{3m}{4} = 
(1 - 8\alp) \frac{3}{4} (n - |A|) \geq 
(1 - 8\alp)(1 - \alp) \frac{3n}{4} 
\geq 
(1 - 9\alp) \frac{3n}{4} = 
(1 - \eps) \frac{3n}{4}, 
$$
then $G$ would be $\eps$-extremal, a contradiction.

Lemma~\ref{almost-perfect} implies that 
$H$ admits a $D$-tiling covering all but $50/\alp$ vertices.  Set $W \subset V(H)$ to be the set of vertices (if any) 
uncovered by this $D$-tiling.  Since 
$|V(H) \setminus W|$ is divisible by 4, it must be that $|A \cup W|$ is divisible by 4, and so Lemma~\ref{absorbing-lemma} guarantees that $G[A \cup W]$ is $D$-tileable.  
Thus, $G$ is $D$-tileable.

\end{proof}

\subsection{Proof of Lemma~\ref{almost-perfect}}    
\label{sec:almost-perfect}
Let $\g > 0$ be given, and let 
$m$ be a sufficiently large integer which is divisible by 4.  
Let $H$ be a 3-graph of order $m$, which is not $(8\g)$-extremal, and for which 
$\del_2(H) \geq\left(\tfrac{1}{4} - \g\right)m$.  
We prove that $H$ contains a $D$-tiling covering all but $50/\g$ vertices.  
To that end, 
let $\cM$ be a maximum $D$-tiling in $H$, but assume, on the contrary, that $\cM$ leaves more than $50/\g$ vertices uncovered.

We use the following notation and terminology.  Let $V_{\cM}$ denote the set of vertices of $H$ covered by $\cM$, and let $W = V(H) \setminus V_{\cM}$.  
For a vertex $v \in V_{\cM}$, write $H_v[W]$ for $N_H(v) \cap \tbinom{W}{2}$, 
and say that 
$v \in V_{\cM}$ is {\it $W$-big} if 
$|H_v[W]| 
\geq 10|W|$, and {\it $W$-small} otherwise.  
Observe that every element $D_0 \in \cM$ contains at most one $W$-big vertex.  Indeed, assuming otherwise, let $u, v\in V(D_0)$ both be $W$-big vertices.
Since 
$|H_u[W]|  \geq 10  
|W| > |W| /2$, 
the graph $H_u[W]$
contains a path of length 2, with vertices denoted by $w_1, w_2, w_3$.  The graph 
$H_v[W\setminus \{w_1, w_2, w_3\}]$ then has size
\begin{equation}
\label{eqn:2path}
\left| H_v\left[W\setminus \{w_1, w_2, w_3\}\right] \right| 
\geq 
|H_v[W]|  - 3 |W| \geq 7 |W| > |W| /2,  
\end{equation}
and so $H_v[W\setminus \{w_1, w_2, w_3\}]$ 
contains a path of length 2, with vertices denoted by $w_1', w_2', w_3'$.  
Then, $\{u, w_1, w_2, w_3\}$ and $\{v, w_1', w_2', w_3'\}$ span vertex-disjoint copies of $D$, which can replace 
$D_0$ in $\cM$ to contradict that $\cM$ was a maximum $D$-tiling in $H$.

Now, write $B$ for the set of $W$-big vertices $v \in V_{\cM}$, and write $|B| = b$.
We now observe that $b \geq \left(\tfrac{1}{4} - 2\g\right)m$.  
Indeed, write $H[W,W,V_{\cM}]$ for the set of triples from $H$ containing exactly two vertices from $W$.  From our definitions above, 
note that 
$$
\left| H[W,W,V_{\cM}] \right| \leq b \left( 30 |W| + \binom{|W|}{2} \right) + 40 (|\cM| - b) |W| \leq 
b \binom{|W|}{2} + 40 |\cM| |W| \leq b \binom{|W|}{2} + 10 m |W|.  
$$
On the other hand, the maximality of $\cM$ implies that $H[W]$ is $D$-free, and so 
$$
\left| H[W,W,V_{\cM}] \right| \geq 
\left(\left(\frac{1}{4} - \g\right)m - 1\right) \binom{|W|}{2}.  
$$
The inequalities above imply that 
$$
b \geq \left(\frac{1}{4} - \g\right) m - 1 - \frac{20 m}{|W| - 1} \geq 
\left(\frac{1}{4} - \g\right) m - 1 - \frac{40 m}{|W|}, 
$$
and by our assumption that $|W| > 50/\g$, we infer that $b \geq \left(\tfrac{1}{4} - 2\g\right)m$, as claimed.

Now, write $\cM_{B} \subset \cM$ for elements of $\cM$ which contain a $W$-big vertex, and let $V_{\cM_B}$  
denote the set of vertices of $H$ covered by $\cM_B$.  
Then, $S_B = V_{\cM_B} \setminus B$ consists of $W$-small vertices and we have 
$|S_B| = 3|B| \geq (1 - 8\g) 3m/4$.  
Since $H$ is not $(8\g)$-extremal, $H[S_B]$ contains a copy $D_0$ of $D$, say with vertices $v_1, v_2, v_3, v_4$.  
Let $u_1, u_2, u_3, u_4$ denote 
the 
$W$-big vertices corresponding to $v_1, v_2, v_3, v_4$, respectively, in $\cM_B$.  Among $u_1, \dots, u_4$, at least two and at most 4 are distinct, and so w.l.o.g.,
let $u_1, \dots, u_j$, for some $j \in \{2, 3, 4\}$, denote the distinct vertices of $u_1, \dots, u_4$.  
For $1\leq i \leq j$, 
let $D_i \in \cM_B$ be the element containing $u_i$.

Similarly to~(\ref{eqn:2path}), 
the definition of a $W$-big vertex will guarantee, for each $1\leq i \leq j$, the existence of a 2-path
$P_2(u_i) \subset H_{u_i}[W]$ so that
$P_2(u_1), \dots, P_2(u_j)$ are each pair-wise vertex-disjoint.  
Indeed, if we already have the desired 2-paths $P_2(u_1), \dots, P_2(u_{i-1})$, where $2\leq i \leq j \leq 4$, then 
$$
\Big|H_{u_i}\Big[W \setminus \big(V(P_2(u_1)) \cup \dots \cup V(P_2(u_{i-1})) \big) \Big] \Big|
\geq 
|H_{u_i}[W]| - 3(i-1)|W|  
\geq |H_{u_i}[W]| - 9|W|  
\geq
|W| > |W|/2, 
$$
and so there exists a 2-path $P_2(u_i) \subset H_{u_i}[W]$ which is vertex-disjoint from each of $P_2(u_1), \dots, P_2(u_{i-1})$.

Clearly, for each $1\leq i \leq j$, $\{u_i\} \cup V(P_2(u_i))$ spans a copy of $D$, which we shall denote as 
$D^{u_i}$.  Then, $D^{u_1}, \dots, D^{u_j}$ are pair-wise vertex-disjoint 
copies of $D$, and so,  
deleting from $\cM$
the elements 
$D_1, \dots, D_j$ and adding $D_0, D^{u_1}, \dots, D^{u_j}$ contradicts that $\cM$ was a maximum $D$-tiling.  
This concludes the proof 
of Lemma~\ref{almost-perfect}.

\subsection{Proof of Lemma~\ref{absorbing-lemma} -- Absorption}  
\label{sec:absorbing-lemma}  
We shall prove the following stronger form of Lemma~\ref{absorbing-lemma}, which allows for a smaller co-degree and larger choices of subset $W$.  

\begin{lemma}[Lemma~\ref{absorbing-lemma} - strong form]  
\label{lem:abs}
For all $\alp, \del > 0$, there exists $\omega > 0$ so that for all sufficiently large integers $n$ divisible by 4, the following holds.  
Let $G$ be a 3-graph of order $n$.  If $\del_2 (G) \geq \del n$, then there exists $A \subset V(G)$ of size $|A| \leq \alp n$ so that, for every 
$W \subset V\setminus A$ of size $|W| \leq \omega n$ for which $|A \cup W|$ is divisible by 4, the hypergraph $G[A \cup W]$ is $D$-tileable.  
\end{lemma}  

Our proof of Lemma~\ref{lem:abs} will be based on 
Proposition~\ref{prop:abs}, for which we need the following definition.  
\begin{definition}  
\label{def:absorb}  
\rm 
Suppose $G$ is a 3-graph with vertex set $V$, and let $U \in \tbinom{V}{4}$.  We say that a set $S \in \tbinom{V \setminus U}{8}$ {\it absorbs}  
$U$ if $G[S]$ is $D$-tileable and $G[S \cup U]$ is $D$-tileable.
\end{definition}

\begin{proposition}
\label{prop:abs}  
For all $\del > 0$, there exists $\sig > 0$
so that for all sufficiently large integers $n$, the following holds.  Suppose $G$ is a 3-graph 
with vertex set $V$ 
of order $|V| = n$ for which $\del_2(G) \geq \del n$.
For each $U \in \tbinom{V}{4}$, there are $\sig n^8$ sets $S \in \tbinom{V}{8}$ which absorb $U$.  
\end{proposition}  



To prove Proposition~\ref{prop:abs}, we require the following well-known `supersaturation' result of Erd\H{o}s~\cite{erdos}
(stated here only in special case form).  

\begin{theorem}[Erd\H{o}s~\cite{erdos}]  
\label{thm:erd}  
For all $c_1 > 0$ there exists $c_2 > 0$ so that for all sufficiently large integers $n$, the following holds.
If $H$ is a 3-graph of order $n$ and size $|H| \geq c_1 n^3$, then $H$ contains at least $c_2 n^9$ copies of 
$K^{3}_{3,3,3}$ (the balanced complete 3-partite 3-graph of order 9).    
\end{theorem}  

\begin{proof}[Proof of Proposition~\ref{prop:abs}]

Let $\del > 0$ be given.  
Let $c_1 = \del^3 / 36$, and let $c_2 > 0$ be the constant guaranteed by Theorem~\ref{thm:erd}.  
We define $\sig = c_2$,  
and in all that follows, we take $n$ to be a sufficiently large integer.  
Let $G$ be a 3-graph with vertex set $V$ of order $|V| = n$ for which $\del_2(G) \geq \del n$.  Fix $U = \{u_1, u_2, u_3, u_4\} 
\subset V$.  We prove there are $\sig n^8$ 
sets $S \in \tbinom{V}{8}$ which absorb $U$.

To that end, define $V_1 = N(u_1, u_2)$, $V_2 = N(u_3, u_4)$ and 
$$
V_3 = \bigcup \big\{N(v_1, v_2): (v_1, v_2) \in V_1 \times V_2 \big\}.  
$$
Note that $V_1 \cup V_2 \cup V_3$ is not necessarily a partition, 
but it will not be difficult to find 
pairwise disjoint subsets $W_i \subset V_i$, $i = 1, 2, 3$, for which 
$|G[W_1, W_2, W_3]| \geq c_1 n^3$.  
To that end, let $W_1 \subset V_1 \setminus \{u_3, u_4\}$ be any set of size (exactly) $\lceil \del n / 3 \rceil$
(recall $|V_1| \geq \del n$).  
Let $W_2 \subset V_2 \setminus (W_1 \cup \{u_1, u_2\})$ be any set of size
(exactly) $\lceil \del n / 3 \rceil$ (recall $|V_2| \geq \del n$).  
Now, set $W_3 = V_3 \setminus (W_1 \cup W_2 \cup \{u_1, u_2, u_3, u_4\})$.  
Then, 
$$
|G[W_1,W_2,W_3]| = \sum_{(w_1, w_2) \in W_1 \times W_2} |N(w_1,w_2) \cap W_3|  
\geq 
\left\lceil \frac{\del n}{3} \right\rceil^2 \left( \del n - 
2\left\lceil \frac{\del n}{3} \right\rceil
- 4
\right)  
\geq 
\frac{\del^3n^3 }{36} = c_1 n^3.  
$$


Now, set $H = G[W_1, W_2, W_3]$, which we view as a hypergraph of order $n$.  Since $H$ has size 
$|H| \geq c_1 n^3$, Theorem~\ref{thm:erd} guarantees that $H$ has at least $c_2 n^9 = \sig n^9$ copies of $K^{3}_{3,3,3}$.  
Note that each such copy has exactly 3 vertices in each of $W_1, W_2, W_3$ and that, for some fixed
$w_3 \in W_3$ (it doesn't matter which), at least $\sig n^8$ such copies 
contain the vertex $w_3$.  
Let $\{w_1, w_1', w_1'', w_2, w_2', w_2'', w_3, w_3', w_3''\}$ denote the vertex set of such a copy, where 
$w_i, w_i', w_i'' \in W_i$, $i = 1, 2, 3$.  
We claim that 
$$
S_U = 
S_U(w_3) = 
\{w_1, w_1', w_1'', w_2, w_2', w_2'', w_3', w_3''\} 
$$
absorbs the set $U$
(see Figure~\ref{fig-absorbing}).  
Indeed, 
$$
S_1:= \Big\{ \{w_1, w_2, w_3'\}, \, \{w_1', w_2, w_3'\} \Big\}\, ,  \quad 
S_2:= \Big\{ \{w_1'', w_2', w_3''\}, \, \{w_1'', w_2'', w_3''\} \Big\}  
$$
is a $D$-tiling of $G[S_U]$ and 
$$
T_1:=\Big\{ \{u_1, u_2, w_1\}, \, \{u_1, u_2, w_1'\} \Big\}\,,  ~ 
T_2:=\Big\{ \{u_3, u_4, w_2\}, \, \{u_3, u_4, w_2'\} \Big\}\,,  ~
T_3:=\Big\{ \{w_1'', w_2'', w_3'\}, \, \{w_1'',w_2'',w_3''\} \Big\}
$$
is a $D$-tiling of $G[S_U \cup U]$.  

\end{proof}

\begin{figure}
\begin{center}
\scalebox{1}{\includegraphics{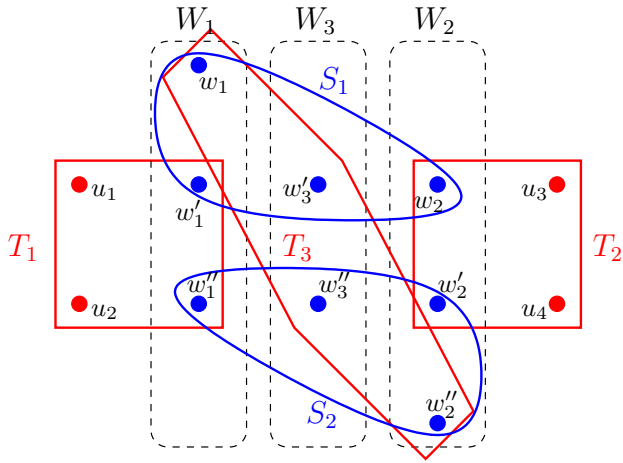}}
\caption{Absorbing structure.}\label{fig-absorbing}
\end{center}
\end{figure}

Finally we use Proposition \ref{prop:abs} to prove Lemma \ref{lem:abs}.

\begin{proof}[Proof of Lemma~\ref{lem:abs}]  
Let $\alp, \del > 0$ be given.  
Let $\sig = \sig (\del) > 0$ be the constant guaranteed by Proposition~\ref{prop:abs}.  
We define
\begin{equation}
\label{eqn:w}
\omega = \frac{\alp \sig^2}{128}.  
\end{equation}
In all that follows, we take $n$ to be a sufficiently large integer divisible by 4.  
Let $G$ be a given 3-graph with vertex set $V$ of order $|V| = n$ for which $\del_2 (G) \geq \del n$.  
We prove that $G$ admits a set $A \subset V$ described in the conclusion of Lemma~\ref{lem:abs}.  To 
produce the desired set $A$, we employ the well-known deletion method in probabilistic combinatorics.

To begin, 
set $p = (1/16) \alp \sig n^{-7}$, and let $\bbH = \bbH^{(8)}(n, p)$ be the binomial random 8-uniform hypergraph with $n$-element vertex set $V$.  
We note several basic properties of $\bbH$ (due to the Chernoff inequality, unless otherwise indicated):  
\begin{enumerate}
\item  
With probability $1 - \exp \{- n/\log n\}$, 
$$
|\bbH| \leq 2 p \binom{n}{8}   \leq \frac{1}{8}\alp n; 
$$
\item  
Let $\bbH \otimes \bbH = \{(S_1, S_2) \in \bbH \times \bbH: S_1 \cap S_2 \neq \emptyset\}$.  
Then, 
$$
\bbE \left[ |\bbH \otimes \bbH| \right] \leq 8\binom{n}{8} \binom{n}{7} p^2 \leq 
\frac{1}{256}  
\alp^2 \sig^2  n.    
$$
As such, by the Markov inequality, 
$$
\Pr \left[ |\bbH \otimes \bbH| \geq \frac{1}{128} \alp^2 \sig^2 n \right] \leq \frac{1}{2};    
$$
\item
For $U \in \tbinom{V}{4}$, let 
$\cA(U)$ be the collection of sets $S \in \tbinom{V}{8}$ which absorb $U$.  
By Proposition~\ref{prop:abs}, $|\cA(U)| \geq \sig n^8$, and so 
with probability $1 - \exp \{- n/\log n\}$, $\bbH$ satisfies that for every $U \in \tbinom{V}{4}$, 
$$
|\cA(U) \cap \bbH| \geq \frac{1}{2} 
p |\cA(U)| 
\geq 
\frac{1}{32}  
\alp \sig^2 n.   
$$
\end{enumerate}

Let $H$ be an instance of $\bbH$ for which properties~$(i)$--$(iii)$ hold (and specifically, where 
$|H \otimes H| < \alp^2 \sig^2 n/ 128$).  
Now, 
\begin{enumerate}
\item[$(a)$]  
delete any $S \in H$ for which there exists $S' \in H$ for which 
$S \cap S' \neq \emptyset$.  This deletes at most 
$$
2 \times \frac{\alp^2 \sig^2 n}{128} = 
\frac{\alp^2 \sig^2 n}{64} 
$$
elements $S \in H$;   
\item[$(b)$]  
delete any $S \in H$ 
for which no $U \in \tbinom{V}{4}$ has $S \in \cA(U)$.  
\end{enumerate}  
The resulting hypergraph is then, importantly, a (partial) matching $M$ in $V$.  Let $m:=|M|$, $\{S_1,\dots, S_m\}=M$, and $A:=\bigcup_{i=1}^mS_i$ (the set of vertices covered by $M$).  We now confirm that $A$ satisfies its claimed properties.  

Observe from~$(i)$ that $|A| = 8 |M| \leq \alp n$, as promised.  
Now, let $W \subset V \setminus A$ have size $4t:=|W| \leq \omega n$ 
(cf.~(\ref{eqn:w})) and then arbitrarily partition $W$ into 4-sets $\{W_1,W_2,\dots, W_{t}\}=:\mathcal{W}$.

Note that by (iii), (a), and \eqref{eqn:w} we have that for all $W_i\in \mathcal{W}$, $$|\mathcal{A}(W_i)\cap M|\geq \frac{1}{32} \alp \sig^2 n -\frac{1}{64} \alp^2 \sig^2 n \geq \frac{1}{64} \alp \sig^2 n\geq \frac{\omega n}{4}\geq t.$$
So for each $W_i\in \mathcal{W}$ we can greedily choose some unique $S_i'\in \mathcal{A}(W_i)\subseteq M$, which guarantees that each of $G[S_1' \cup W_1], \dots, G[S_t' \cup W_t]$ are $D$-tileable.  Finally, since $G[S]$ is $D$-tilable for all $S\in M$ (by (b) and Definition \ref{def:absorb}), and since $\{S_1,\dots, S_m, W_1,\dots, W_t\}$ is a partition of $A\cup W$, we infer that $G[A\cup W]$ is $D$-tileable as desired.

\end{proof}

\subsection*{Acknowledgements}

Our thanks to the referees for their thoughtful suggestions.

\end{document}